\documentclass[a4paper]{article}

\usepackage[english]{babel}
\usepackage[utf8x]{inputenc}

\usepackage[a4paper,top=3cm,bottom=2cm,left=3cm,right=3cm,marginparwidth=1.75cm]{geometry}

\usepackage{amssymb,amsmath,amsthm,amsfonts,graphicx,tikz,subfig,rotating,tabularx, enumerate, tabu}
\usetikzlibrary{positioning,calc}

\newtheorem{theorem}{Theorem}[section]
\newtheorem{proposition}[theorem]{Proposition}
\newtheorem{lemma}[theorem]{Lemma}
\newtheorem{note}[theorem]{Note}
\theoremstyle{definition}
\newtheorem{definition}[theorem]{Definition}
\newtheorem{corollary}[theorem]{Corollary}
\newtheorem{remark}[theorem]{Remark}

\newcommand{\forb}{\textup{forb}}

\newcommand{\avoid}{\textup{Avoid}}
\newcommand{\sym}{\textup{Sym}}
\newcommand{\Sym}{\textup{Sym}}
\newcommand{\ch}{\textup{ch}}

\newcommand{\supp}{\textup{supp}}
\newcommand{\ex}{\textup{ex}}

\newcommand{\zeroone}{\begin{bmatrix}0 & 1\end{bmatrix}}

\title{Multi-Symbol Forbidden Configurations}
\author{Keaton Ellis$^a$, Baian Liu$^a$, Attila Sali$^{b,c,}$\thanks{Research partially supported by the
    National Research, Development and Innovation Office (NKFIH)
    grant K--116769. This work is also connected to the scientific program of the "Development of quality-oriented and harmonized R+D+I strategy and functional model at BME" project, supported by the New Hungary Development Plan (Project ID: T\'AMOP-4.2.1/B-09/1/KMR-2010-0002).
  }\\
  $^a$Budapest Semesters in Mathematics\\
  $^b$Alfr\'ed R\'enyi Institute of Mathematics, Budapest\\ 
$^c$Department of Computer Science and Information Theory,\\ Budapest University of Technology and Economics}

\begin{document}
\maketitle

%
%
%

\begin{abstract}
An $r$-matrix is a matrix with symbols in $\{0,1,\dots,r-1\}$. A matrix is simple if it has no repeated columns. For a family of $r$-matrices $\mathcal{F}$, we define $\forb(m,r,\mathcal{F})$ as the maximum number of columns of an $m$-rowed, simple $r$-matrix $A$ such that $F$ is not a row-column permutation of a submatrix of $A$ for all $F \in \mathcal{F}$. We investigate the asymptotic bounds of special families of forbidden configurations $\Sym(F)$, a symmetric family of matrices based on a $(0,1)$-matrix $F$. The previously known lower bound constructions required $F$ to have no constant row. We introduce a new lower bound construction that drops this requirement and is a better bound in certain cases. We will also investigate the effects of constant rows from the upper bound perspective, including upper bounds of block matrices, matrices consisting of entirely constant rows. 

\end{abstract}

\section{Introduction}
A $(0,1)$-matrix $A$ is  \emph{simple} if contains no repeated columns. For such a matrix $A$, $|A|$ denotes the number of columns in $A$. Suppose a $k \times \ell$ $(0,1)$-matrix $F$ (not necessarily simple) is  given. $F$ is said to be a \emph{configuration} of $A$, denoted by $F \prec A$, if there exists a submatrix of $A$ that is a row and column permutation of $F$. 

One natural question to ask about forbidden configurations is: given an $m$-rowed simple matrix $A$ and a configuration $F$, what is the maximum number of columns $A$ can have if $F \not \prec A$? To put this question into formal notation, define
\[
\avoid(m, F) = \{A : A \text{ is a simple matrix with $m$ rows}, F \not \prec A\},
\]
and the main function to be computed is
\[
\forb(m, F) = \max\{|A| : A \in \avoid(m, F)\}.
\]

The definitions of $\avoid$ and $\forb$ can be extended to accommodate a family of forbidden configurations $\mathcal{F}$ as follows.
\[
\avoid(m, \mathcal{F}) = \{A : A \text{ is a simple matrix with $m$ rows}, F \not \prec A \text{ for all } F\in\mathcal{F}\},
\]
\[
\forb(m, \mathcal{F}) = \max\{|A| : A \in \avoid(m, \mathcal{F})\}.
\]
Another generalization of forbidden configurations we are interested in is the extension from $(0,1)$-matrices to matrices with $r$ symbols, which we call $(0,1,...,r-1)$-matrices. We will use the notations: $\avoid(m, r, \mathcal{F})$ and $\forb(m, r, \mathcal{F})$. Analogously, an $(i,j)$-matrix is matrix of entries from $\{i,j\}$.

Let $F$ be a  $(0,1)$-matrix, define $F(i,j)$ as the $(i,j)$-matrix given by replacing the 0's in $F$ with $i$'s and the 1's in $F$ with $j$'s. Furthermore, let
\[
\sym(F) = \{F(i,j) : 0 \leq i < j \leq r-1\}.
\]	
A similar set is
\[
\mathbf{S}(F) = \{F(i,j) : 0 \leq i, j \leq r-1, i \neq j\}.
\]
Notice that $\mathbf{S}(F) = \sym(F) \cup \sym(F^c)$, where $F^c$ is the $(0,1)$-complement of $F$. 

The importance of $\sym(F)$ and $\mathbf{S}(F)$ is that for $\forb(m,r,\mathcal{F})$, if there exists a pair $(i,j) \in \binom{[r]}{2}$ such that no matrix in $\mathcal{F}$ is an $(i,j)$-matrix, then $K_m(i,j) \in \avoid(m,r,\mathcal{F})$, where $K_m$ is the simple $(0,1)$-matrix of $m$ rows and all possible $2^m$ columns. Thus $\forb(m,r,\mathcal{F}) \geq 2^m$. Otherwise, $\forb(m,r,\sym(F))$ is of polynomial order of magnitude, consequently $\forb(m,r,\mathbf{S}(F))$ is polynomial, as well, as it was shown in \cite{Multivalued}.

The following is a classical and celebrated result in Forbidden Configurations:

\begin{theorem}\label{Sauer's Thm}\cite{Sauer}\cite{Shelah}\cite{VC}
\[
\forb(m,K_k) = \binom{m}{k-1} + \binom{m}{k-2} + \cdots + \binom{m}{1} + \binom{m}{0}.
\]
\end{theorem}

%

Theorem \ref{Sauer's Thm} was extended to $r$ symbols.

\begin{theorem}\cite{Multivalued}\label{extendedsauer}
$\forb(m, r, \sym(K_k)) = \Theta(m^{(k-1)\binom{r}{2}})$.
\end{theorem}

For any $F$, there exists a $k \in \mathbb{N}^+$ such that $\forb(m,r,\sym(F)) = O(m^{(k-1)\binom{r}{2}})$. This is due to the fact that $F \prec K_k$ for some $k \in \mathbb{N}^+$ and the upper bound from Theorem \ref{extendedsauer}. 

\begin{remark}
	Like in the $(0,1)$ case, we get by reversing the order of the symbols,
	\[
	\forb(m,r,\Sym(F)) = \forb(m,r,\Sym(F^c)).
	\]
\end{remark}

In the present paper, we investigate $\forb(m,r,\sym(F))$. The organization is as follows. Section~2 contains lower bounds given by constructive methods. Section~3 deals with general (asymptotic) upper bounds, while Section~4 applies the results of Section~3 to particular forbidden configurations. Section~5 contains cases of block matrices, matrices consisting of blocks of constant matrices. Finally, Section~6 deals with possible future research directions. In what follows $r\ge 3$ is assumed. 

\section{Lower Bound Constructions}
Let $I_\ell$ denote the $\ell\times\ell$ identity matrix and $T_\ell$ denote the $\ell\times\ell$ upper triangular matrix with 1's on and above the diagonal and 0's below the diagonal.
We  use results from $(0,1)$-matrices applying the idea of direct products, that was introduced in \cite{AGS97}, where  $I_\ell$ and  $T_\ell$ were used extensively. The \emph{direct product} $A \times B$ is defined to be a matrix with every column of $A$ placed on top of every column of $B$ in every possible way. Thus if $A$ is a $m \times n$ matrix and $B$ is a $m' \times n'$ matrix, then $A \times B$ is a $(m + m') \times (n \cdot n')$ matrix. The following theorem uses the direct product to provide a construction for the lower bound of $\forb(m, r, \Sym(F))$ if $F$ is simple and has no constant row. 

\begin{theorem}{\cite{Multivalued}}\label{prod} Let $F$ be a simple $(0,1)$-matrix with no constant row. Then
\[\forb(m, r, \Sym(F)) \geq \forb\left(\frac{m}{\binom{r}{2}}, F\right)^{\binom{r}{2}}.\]
\end{theorem}
This lower bound is given by taking $A \in \avoid\left(\frac{m}{\binom{r}{2}}, F\right)$ such that $|A| = \forb\left(\frac{m}{\binom{r}{2}}, F\right)$. Then
\[
\prod_{0 \leq i < j \leq r - 1} A(i,j) \in \avoid(m, r, \Sym(F)).
\]
Due to the form of the construction, this is referred to as the ``product construction''.
\begin{remark}
If $F$ is not simple or has a constant row, but there exists a submatrix $F' \prec F$ that is simple and has no constant row, then we can say
\[\forb(m, r, \Sym(F)) \geq \forb(m, r, \Sym(F')) \geq \forb\left(\frac{m}{\binom{r}{2}}, F'\right)^{\binom{r}{2}}.\]
\end{remark}
This fact is due to the following lemma:

\begin{lemma}\cite{survey}
If $F' \prec F$, then $\forb(m, F') \leq \forb(m, F)$.
\end{lemma}
The next theorem provides lower bounds for forbidden configurations that are not simple or have a constant row. 
\begin{theorem}
\label{lbound}
Let $F$ be a (0,1)-matrix, and let $n_i$ denote the maximum number of $i$'s in a column of $F$, and set $n := \max\{n_0, n_1\}$. Then, if $n \geq 2$, 
\[\forb(m, r, \sym(F)) \geq \sum_{k=0}^{(n-1)(r-1)}\binom{m}{k}\sum_{\substack{k_1,...,k_{r-1}<n, \\ k_1+\cdots + k_{r-1} = k}}\binom{k}{k_1,...,k_{r-1}} = \Omega(m^{(n-1)(r-1)}).\] 
\end{theorem}

\begin{proof}

If $n = n_1$, we consider the total number of each non-zero symbol $j$ in a column. If, for each non-zero symbol, we have less than $n$ entries in a column, we cannot have $F(i,j)$. Therefore we can have all columns with each non-zero symbol appearing at most $n-1$ times, which provides this lower bound. A symmetric argument can be made if $n = n_0$ considering the number of  symbols not equal to $r$ in a column. 
\end{proof}

\begin{remark}
If $F$ has $k$ rows and $k\geq 3$, we have that $\forb(m,r,\Sym(F)) = \Omega(m^{(\lceil\frac{k}{2}\rceil-1)(r-1)})$ since every column has at least $\lceil\frac{k}{2}\rceil$ 0's or at least $\lceil\frac{k}{2}\rceil$ 1's.
\end{remark}

\begin{remark}
If $\forb(m, F) = O(m^k)$ and $n > k\frac{r}{2}$, Theorem \ref{lbound} provides a better asymptotic lower bound than Theorem \ref{prod}. For example, let
\[
\Gamma = \begin{bmatrix}
1 & 1 \\
1 & 0 \\
1 & 0 \\
\end{bmatrix}.
\]
Then the largest subconfiguration without constant rows is $\Gamma' = \begin{bmatrix}
1 & 0 \\ 1 & 0
\end{bmatrix}\prec\Gamma$ with $\forb(m,\Gamma')=\Theta(m)$.
Subsequently, we obtain that $\forb(m,r,\Sym(\Gamma))= \Omega(m^{\max\{2(r-1), \binom{r}{2}\}})$, so the asymptotic lower bound follows a different form for $r=3$. In fact,  $\forb(m,r,\Sym(\Gamma))= \Theta(m^{\max\{2(r-1), \binom{r}{2}\}})$ is obtained as a direct result of Theorem~\ref{1111K2}.
\end{remark}

\section{Upper Bounds}
Theorem \ref{extendedsauer} requires a large $k \in \mathbb{N}^+$ to ensure that the repeated columns are dealt with if non-simple forbidden configurations are studied. However, in most of the cases it is enough to consider the support of a configuration. 

Let $\mu(\textbf{x},F)$ be the multiplicity of column $\textbf{x}$ in matrix $F$. The \emph{support} of a matrix $F$, denoted $\supp(F)$, is the matrix containing all columns such that $\mu(\textbf{x},F) > 0$, \cite{AS14}. 
\begin{theorem}
\label{multiplicity}
Let $F$ be a $k \times \ell$ $(0,1)$-matrix, $s = |\supp(F)|$, and $\mu = \max\limits_{\textbf{x}\in F}(\mu(\textbf{x},F))$. Then
\[
\forb(m,r,\sym(F)) \leq \forb(m,r,\sym(\supp(F))) + \binom{m}{k}\binom{r}{2}(\mu-1)k!s. 
\]
\end{theorem}
\begin{proof}
To prove this, we extend an old idea of F\"uredi's to $r$-symbols. We first begin with $\forb\left(m,r,\Sym\left(\supp(F)\right)\right)+1$ columns. From this, we know that we have a $\supp(F(i_1,j_1))$. Let us permute these $s$ columns containing $\supp(F(i_1,j_1))$ to the first $s$ columns. If we disregard these columns, and add $s$ more to the end of the matrix, we will again have some $\supp(F(i_2,j_2))$. If we add $[\binom{m}{k}\binom{r}{2}(\mu-1)k!]s + 1$ columns instead of just $s$, then we will force at least $\mu$ number of some $\supp(F(i,j))$ in the same $k$ rows, each with the correct row permutation. Since $F \prec \mu \cdot \supp(F)$, we prove our result. 
\end{proof}
This theorem tells us that adding repeated columns in $F$ do not affect the asymptotic bounds of $\forb(m,r,\Sym(F))$ except for a few small cases. In other words, for the purpose of studying asymptotic bounds of $\forb(m,r,\Sym(F))$, we only need to look at $\forb(m,r,\Sym(\supp(F)))$

\begin{corollary}\label{cor:supp-enough}
Let $r\geq 4$. If $\supp(F) \not\prec \zeroone$ or $\supp(F) \not\prec \begin{bmatrix}
0 \\ 1\end{bmatrix}$, then \[\forb(m,r,\sym(F)) = \Theta(\forb(m,r,\sym(\supp(F)))).\]
\end{corollary}

\begin{proof}
Clearly, $\forb(m,r,\sym(F)) = \Omega(\forb(m,r,\sym(\supp(F))))$.

In order to show $\forb(m,r,\sym(F)) = O(\forb(m,r,\sym(\supp(F))))$, we need $\binom{m}{k}\binom{r}{2}(\mu-1)k!s = O(\forb(m,r,\sym(\supp(F))))$ from Theorem \ref{multiplicity}, where $k$ is the number of rows in $F$, $\mu$ is the maximum column multiplicity of $F$, and $s = |\supp(F)|$. 

We see that $\binom{m}{k}\binom{r}{2}(\mu-1)k!s = O(m^k)$. If $k\ge 3$, then we can use Theorem \ref{lbound} to get $\forb(m,r,\sym(\supp(F))) = \Omega(m^{(\lceil\frac{k}{2}-1)\rceil(r-1)}) = \Omega(m^k)$ assuming $r \geq 4$. Thus $\binom{m}{k}\binom{r}{2}(\mu-1)k!s = O(\forb(m,r,\sym(\supp(F))))$, as desired. Therefore, $\forb(m,r,\sym(F)) = \Theta(\forb(m,r,\sym(\supp(F))))$.

Now if $k=2$, then either we have two 0's or two 1's in a column of $\supp(F)$ and we apply Theorem~\ref{lbound} again using that $r-1\ge k$. Otherwise,  $\supp(F) = I_2$ and we use Theorem \ref{prod} to get $\forb(m,r,\sym(\supp(F))) = \Omega(m^{\binom{r}{2}})$. Thus we do get $\binom{m}{k}\binom{r}{2}(\mu-1)k!s = O(\forb(m,r,\sym(\supp(F))))$ from Theorem \ref{multiplicity} since $\binom{m}{k}\binom{r}{2}(\mu-1)k!s = O(m^2)$ and $\forb(m,r,\sym(\supp(F))) = \Omega(m^2)$. 
\end{proof}
The cases $\supp(F) \prec \zeroone$ and $\supp(F) \prec \begin{bmatrix}
0 \\ 1
\end{bmatrix}$ will be treated in Section \ref{sec:blocks}. 

\begin{note}
For $r=3$ $\forb(m,r,\sym(\supp(F)))  = \Omega(m^k)$ is not necessarily true as shown in Theorem~\ref{thm:4.3} and Theorem~\ref{onesontopofzeroes} with $p_0=p_1$.
However, as soon as $F$ has a column with at least $\lfloor\frac{k}{2}\rfloor+1$ 0's or 1's, then Theorem~\ref{lbound} gives the right lower bound. So the question remains what happens if $r=3$ and $F$ is \emph{balanced}, that is the number of 0's and 1's differ by at most one in each column?
\end{note}
The following is a further extension to Theorem \ref{extendedsauer} using Theorem \ref{multiplicity}:
\begin{theorem}\label{extendedextendedsauer}
$
\forb(m,r,\Sym(s \cdot K_k)) = \Theta(m^{(k-1)\binom{r}{2}}). 
$
\end{theorem}

\begin{proof}
The lower bound is given by Theorem \ref{prod} using $K_k \prec s \cdot K_k$, while the upper bound is given by Theorem \ref{multiplicity} since $(k-1) \binom{r}{2} \geq k$ for all $r \geq 3$ and all $k$. 
\end{proof}
\begin{remark}\label{rem:k-1 case}
If $F$ is a $k \times l$ $(0,1)$-matrix with no constant row, and $\forb(m,\supp(F)) = \Theta(m^{k-1})$, then $\forb(m,r,\Sym(F)) = \Theta(m^{(k-1)\binom{r}{2}})$. This follows from Theorem \ref{prod} and Theorem \ref{extendedextendedsauer}. 
\end{remark}

\subsection{Standard Induction}
A basic technique used to find an upper bounds for the $\forb$ function of a given matrix (or family) is standard induction, introduced in \cite{Induction}. Here, we show a known inequality bounding the growth of $\forb$ with respect to $m$.

Let a $k \times \ell$ (0,1)-matrix $F$ be given. If $A$ is such that $A \in \avoid(m,F)$ and $|A| = \forb(m,F)$, then let us consider some row $t$ of $A$. By permuting $t$ to the top row and then permuting the columns, we get the following permutation of $A$:
\[
A = \begin{bmatrix}
0 & \dots & 0 & 1 & \dots & 1\\
B_t & \dots & C_t & C_t & \dots & D_t
\end{bmatrix}
\]
where $C_t$ denotes the columns that appear both under 0 and 1, while $B_t$ and $D_t$ are columns that are only under a  single symbol. 

$C_t$ needs to avoid a simpler family of matrices than $A$ does. To describe exactly what the simpler family of matrices is, we need the following notation:
\begin{definition}
Let $F$ be a (0,1)-matrix. Denote
\[
\ch(F) = \{ G \mid G \prec F \prec (G \times \begin{bmatrix}
0 & 1
\end{bmatrix}), \forall H : H \prec G, H \neq G \implies F \not\prec (H \times \begin{bmatrix}
0 & 1
\end{bmatrix})\}.
\]
And its repeated application:
\[
\ch^n(F) = \bigcup_{G \in \ch^{n-1}(F)} \ch(G), \textup{ where } \ch^1(F) = \ch(F).
\]
\end{definition}

\begin{remark}
$\forb(m,F) \leq \forb(m-1,\ch(F))+\forb(m-1,F)$
\end{remark}
Assume that $A\in\avoid(m,F)$ with $|A|=\forb(m,f)$.
Note that we have 
$\begin{bmatrix}B_t\ C_t\ D_t\end{bmatrix} \in \avoid(m-1,F)$.
It is clear to see now that
\begin{equation*}
\forb(m,F)=|A|=|C_t|+|\begin{bmatrix}
B_t\ C_t\ D_t
\end{bmatrix}| \leq \forb(m-1,\ch(F))+\forb(m-1,F).
\end{equation*}
Therefore if we are able to limit the size of $\forb(m-1,\ch(F))$, we can get an upper bound for $\forb(m,F)$ using induction.

When we extend this concept to $r$ symbols, we may possibly have repeated columns under multiple symbols, let $B_t^{i_1,i_2,\dots,i_b}$ denote the collection of columns appearing under exactly each of the symbols $\{i_1,i_2,\dots,i_b\}$ for some subset of our $r$ symbols. We still take $A \in \avoid(m,r,\mathcal{F})$ with $|A| = \forb(m,r,\mathcal{F})$ and a row $t$, and have the matrix:

\[A =
\begin{bmatrix}
0 & \cdots & \cdots & \cdots & \cdots & \cdots & 0 & \cdots & r-1 & \cdots & r-1 \\ B_t^0 & B_t^{0,1} & \cdots & B_t^{0,r-1} & B_t^{0,1,2} & \cdots & B_t^{0,1,\cdots, r-1} & \cdots & B_t^{r-1} & \cdots B_t^{0,1,\cdots, r-1} 
\end{bmatrix}.
\]
Note, that for $B_t^0=B_t$, $B_t^{0,1} = C_t$ and $B_t^1=D_t$ in case of $r=2$.
To make working with the standard induction decomposition easier in $r$ symbols, consider the a slightly larger matrix by changing the definition of the repeated columns. Let $C_t^{i,j}$ be the repeated columns under the symbols $i$ and $j$ in row $t$. 

\[
A' =
\begin{bmatrix}
0 & \cdots & \cdots & 0 & \cdots & r-1 & \cdots & \cdots & r-1 \\
B_t^0 & C_t^{0,1} & \cdots & C_t^{0,r-1} & \cdots & B_t^{r-1} & C_t^{0,r-1} & \cdots & C_t^{r-2,r-1}
\end{bmatrix}.
\]
That is, $A'$ is composed of columns of $\begin{bmatrix}i\\B_t^i\end{bmatrix}$ for all $i=0,1,\ldots ,r-1$ and $\begin{bmatrix}i&j\\C_t^{i,j}&C_t^{i,j}\end{bmatrix}$ for all $0\le i<j\le r-1$.
Note that we do over count here, as a column that is shared by 0, 1, and 2 will be counted in all three $C_t^{0,1}$, $C_t^{0,2}$ and $C_t^{1,2}$. However, we do not worry about the growth of $B_t^{i,j,k}$, as $B_t^{i,j,k} \prec C_t^{i,j}$. Therefore if we restrict $C_t^{i,j}$, then we are restricting columns appearing under multiple symbols as well, and will be able to restrict $A'$. Since $|A'| \geq |A| = \forb(m,r,\mathcal{F})$, by obtaining an upper bound for $A'$ we also obtain one for $\forb(m,r,\mathcal{F})$. 

\begin{remark} To make more formal that the growth of $\forb(m,r,\mathcal{F})$ is bounded by the size of the repeated columns, consider
	\begin{align*}
	\forb(m,r,\mathcal{F}) = |A| \leq |A'| &= \big(|B_t^0|+ |B_t^1| + \cdots + |B_t^{r-1}|\big) + 2\big(|C_t^{0,1}| + |C_t^{0,2}| + \cdots + |C_t^{r-2,r-1}|\big) \\
	& \leq \forb(m-1,r,\mathcal{F}) + 2\big(|C_t^{0,1}| + |C_t^{0,2}| + \cdots + |C_t^{r-2,r-1}|\big).
	\end{align*}
	
Thus,  
\[\forb(m,r,\mathcal{F}) - \forb(m - 1,r,\mathcal{F}) \leq 2\big(|C_t^{0,1}| + |C_t^{0,2}| + \cdots + |C_t^{r-2,r-1}|\big), \]
so if $|C_t^{0,1}| , |C_t^{0,2}| , \dots , |C_t^{r-2,r-1}| \in O(f(m))$ for some monotone increasing function $f\colon \mathbb{N}\rightarrow\mathbb{N}$, then $\forb(m,r,\mathcal{F}) = O(mf(m))$. 
\end{remark}

\subsubsection{Multiple Inductions}\label{multiinductions}
Since $C_t^{i,j}$ needs to avoid $\mathcal{F}' = \{F'|F \prec F'\times \begin{bmatrix}i&j\end{bmatrix} \, \forall F \in \mathcal{F}\}$, in addition to inducting once, we can induct multiple times until we reach a forbidden family of matrices $\mathcal{F}'$ such that $\forb(m,r,\mathcal{F}')$ is already known or can be calculated. In order to talk about this idea, we need to introduce some notation.

\begin{definition}
	Let $A_F^{(ij)}(t)$ be the repeated columns of $A\in \avoid(m,r,\sym(F))$ with $|A| = \forb(m,r,\sym(F))$ under the symbols $i$ and $j$ in row $t$. This corresponds to $C_t^{ij}$ in the previous section.
	
	Now we recursively define $A_F^{(i_1j_1)(i_2j_2)\cdots(i_nj_n)}(t_1,t_2,...,t_n)$ as the repeated columns of\\ $A_F^{(i_1j_1)(i_2j_2)\cdots(i_{n-1}j_{n-1})}(t_1,t_2,...,t_{n-1})$ under the symbols $i_n$ and $j_n$ in row $t_{n-1}$. 
	
	For brevity, we will sometimes shorten $A_F^{(i_1j_1)(i_2j_2)\cdots(i_nj_n)}(t_1,t_2,...,t_n)$ to $A_F^{(i_1j_1)(i_2j_2)\cdots(i_nj_n)}$, where the rows are assumed to be chosen beforehand.
\end{definition}

Given $t_1, t_2, ... , t_n$, if we can bound $|A_F^{(i_1j_1)(i_2j_2)\cdots(i_nj_n)}(t_1,t_2,...,t_n)| = O(f(m))$ for all possible sequences of symbol pairs $(i_1j_1),(i_2j_2),...,(i_nj_n)$, then $\forb(m,r,\Sym(F)) = O(m^nf(m))$. This is due to the fact that the standard induction was done $n$ times, and backtracking the recursion yields an extra factor of $m^n$. 

\subsubsection{"Simplifying" Multi-Symbols with the Induction Multigraph}
It is also useful to visualize this repeated induction as a multigraph of the inductions we have done. Because we do not distinguish between a simple graph and a multigraph in this paper, we use the term ``graph'' to refer to both of them. We first want to turn our sequence of symbol pairs into a graph. 

\begin{definition}
	Let $\textup{Graph}((i_1j_1),(i_2j_2),...,(i_nj_n))$ denote the graph with\\ $\bigcup_{k=1}^n\{i_k, j_k\}$ as vertices and multiset $\{i_1,j_1\},\{i_2,j_2\},...,\{i_n,j_n\}$ as edges. 
\end{definition}

The language of graphs is useful in this context because certain subsequences of symbol pairs contained in the sequence of symbol pairs in $A_F^{(i_1j_1)(i_2j_2)\cdots(i_nj_n)}(t_1,t_2,...,t_n)$ allow us to automatically bound $|A_F^{(i_1j_1)(i_2j_2)\cdots(i_nj_n)}(t_1,t_2,...,t_n)|$. 
 
\begin{definition} For a given matrix $F$, define the family of forbidden submultigraphs as 
	\begin{align*}
\mathcal{H}_F = &\, \{G :\ \forall t_1,...,t_n \text{ distinct}, \forall A_F \in \avoid(m,r,\sym(F)), \\&\, \text{if } \exists (k_1, \ell_1), (k_2, \ell_2), ..., (k_q, \ell_q) \text{ a subsequence of } (i_1j_1),(i_2j_2),...,(i_nj_n)\\&\, \text{such that } \textup{Graph}((k_1, \ell_1), (k_2, \ell_2), ..., (k_q, \ell_q)) \cong G, \\&\,
\text{then }  |A_F^{(i_1j_1)(i_2j_2)\cdots(i_nj_n)}(t_1,t_2,...,t_n)| = O(1) \}.		
\end{align*}
\end{definition}

 We note that this $\mathcal{H}_F$ is infinite, as if some $G$ is in $\mathcal{H}_F$, then any supergraph of $G$ is also. We want to determine $\mathcal{H}_F$ as completely as possible to help us bound $\forb(m,r,\Sym(F))$. 

\begin{definition}
	We say that a matrix is $p$-simple if it has maximum column multiplicity of $p$. For a simple graph $G$, $p\cdot G$ denotes the $p$-simple graph $G'$ obtained by replacing each edge of $G$ by $p$ parallel edges.
\end{definition}

The first forbidden subgraph result is based on the $p$-simplicity of $F$.

\begin{lemma}\label{klog2p}
Let $F$ be a $k \times \ell$, $p$-simple, $(0,1)$-matrix, then $\textup{Graph}(( k+\lceil\log_2(p)\rceil)\cdot (i j))$ is in $\mathcal{H}_F$. 
\end{lemma}

\begin{proof}
Let $A_F \in \avoid(m,r,\Sym(F))$ and $q=k+\lceil\log_2(p)\rceil$. Consider $A_F^S$ with $q\cdot (i j)$ a subsequence of $S$. Suppose for a contradiction that $|A_F^S| > 0$, then
\[
F(i,j) \prec p\cdot K_k(i,j) \prec
\begin{bmatrix}
i&j
\end{bmatrix}^{k+\lceil\log_2(p)\rceil} \prec
\prod_{(xy) \in S} \begin{bmatrix}
x&y
\end{bmatrix} \prec A_F,
\]
which is a contradiction. This implies then that $|A_F^S| = 0 = O(1)$, so $\textup{Graph}(q\cdot (i j)) \in \mathcal{H}_F$. 
\end{proof}

The following is based on a definition from Brown and Simonovits \cite{BS99}.
\begin{definition}
Let us have a family of graphs $\mathcal{F}$. Let $\mathcal{G}_{n,q}$ be the set of all multigraphs on $n$ vertices with maximum edge multiplicity $q$ and let $H \prec G$ denote that $H$ is a submultigraph of the multigraph $G$. Then we define
\[
\ex_q(n,\mathcal{F}) = \max\{|E(G)| : G\in \mathcal{G}_{n,q} \text{ and }\forall F \in \mathcal{F},  \not\exists H\prec G, H\cong F\}.
\]
This takes the maximum number of edges of a graph that does not contain any $F \in \mathcal{F}$ as a subgraph. 
\end{definition}
\begin{lemma}
\label{multisymbol to multicolumn}
Let $F$ be a $k \times \ell$, $p$-simple, $(0,1)$-matrix, and let $q = k+\lceil\log_2(p)\rceil-1$. Then \[\forb(m,r,\sym(F)) = O\big(m^{1+\ex_{q}(r,\mathcal{H}_F)}\big).\]
\end{lemma}

\begin{proof}
Let $A \in \avoid(m,r,\Sym(F))$. Let $S$ be a sequence of symbol pairs of length $1+\ex_q(r, \mathcal{H}_F)$. Then $A_F^S$ is forced by the definition of $\ex_q$ to have a symbol pair repeated $q+1$ times or a subsequence $S'$ such that $\textup{Graph}(S')$ is isomorphic to some graph in $\mathcal{H}_F$. In either case $|A_F^S| = O(1)$ for all $S$ by the definition of $\mathcal{H}_F$. Since $S$ has length $1+\ex_q(r, \mathcal{H}_F)$, we remarked earlier that this implies that $\forb(m,r,\sym(F)) = O\big(m^{1+\ex_{q}(r,\mathcal{H}_F)}\big)$.
\end{proof}

\begin{remark}
	For a subset $\mathcal{H}'_F \subseteq \mathcal{H}_F$ using $\ex_q(r,\mathcal{H}_F) \leq \ex_q(r,\mathcal{H}'_F)$, we have that $\forb(m,r,\sym(F)) = O\big(m^{1+\ex_{q}(r,\mathcal{H}'_F)}\big)$. Thus we do not need to fully determine $\mathcal{H}_F$ to obtain an upper bound. 
\end{remark}

Now that we have established the multigraph framework, we wish to determine what graphs are in $\mathcal{H}_F$ given a forbidden configuration $F$. 

\subsection{Determining $\mathcal{H}_F$}\label{Section: Restrictions of the Induction Multigraph}
Motivated by Lemma~\ref{multisymbol to multicolumn}, we study $\mathcal{H}_F$ for a given configuration $F$. The following comes directly from the definition of $\ch^n(F)$. 
\begin{remark}\label{rem:nonkr}
If $\forb(m,r,\sym(\ch^n(F))) = O(1)$, then $n \cdot K_r \in \mathcal{H}_F$, where $K_r$ denotes the complete graph on $r$ vertices.
\end{remark}
Since $K_r$ is the usual notation for the complete graph on $r$ vertices, we use it whenever we speak about graphs in $\mathcal{H}_F$. This should not cause confusion with $K_k$ that denotes the complete configuration on $k$ rows. 
The next two lemmas are used for proving constant upper bounds. 
\begin{lemma}\cite{ansteelu}\cite{bb}\label{al14} For any $\ell \geq 2$, 
$\forb(m, r, \mathbf{S}(I_\ell, T_\ell)) \leq 2^{c_r\ell^2}$ for some constant $c_r$.  
\end{lemma}

\begin{definition}
Let $\mathcal{F} = \{F_1, F_2, ... , F_k\}$ and $\mathcal{G} = \{G_1, G_2, . . . , G_\ell\}$. If for every $G_i$, there is some $F_j$ with $F_j \prec G_i$, we say that $\mathcal{F} \prec \mathcal{G}$.
\end{definition}

\begin{lemma}\cite{fing}\label{fingtheorem}
If $\mathcal{F} \prec \mathcal{G}$, then $\forb(m, r, \mathcal{F}) \leq \forb(m, r, \mathcal{G})$.
\end{lemma}
For the sake of convenience we introduce the following notations.
\begin{definition}
$\Gamma_n$, $\mathfrak{C}_n$, and $\mathfrak{D}_n$ are defined to be the following graphs.

\[\begin{tikzpicture}
\node[draw,circle] at (0,0)(00){};
\node[draw,circle] at (0,-2)(0-2){};
\node[draw,circle] at (2,0)(20){};
\node[draw,circle, color=white] at (2,-2)(2-2){};
\node at (0,-1)(...){\dots};
\node at (0,-0.75)(n){$n$};
\draw[bend left=20] (00) to (0-2);
\draw[bend right=20] (00) to (0-2);
\draw[bend left=30] (00) to (0-2);
\draw[bend right=30] (00) to (0-2);
\draw[bend left=40] (00) to (0-2);
\draw[bend right=40] (00) to (0-2);
\draw[bend left] (00) to (20);
\draw[bend right,color=white] (0-2) to (2-2);
\end{tikzpicture}\hspace{1cm} \begin{tikzpicture}
\foreach \s in {0,2}{
	\foreach \t in {0,-2}{
    
	\node[draw,circle] at (\s,\t)(\s\t){};
    }
    }
\node at (0,-1)(...){\dots};
\node at (0,-0.75)(n){$n$};
\draw[bend left=20] (00) to (0-2);
\draw[bend right=20] (00) to (0-2);
\draw[bend left=30] (00) to (0-2);
\draw[bend right=30] (00) to (0-2);
\draw[bend left=40] (00) to (0-2);
\draw[bend right=40] (00) to (0-2);
\draw[bend left] (00) to (20);
\draw[bend right] (0-2) to (2-2);
\node[draw, circle] at (4,0)(40){};
\node[draw, circle] at (4,-2)(4-2){};
\node[draw, circle] at (5,-1)(5-1){};
\node at (4,-1)(...){\dots};
\node at (4,-0.75)(n){$n$};
\draw[bend left=20] (40) to (4-2);
\draw[bend right=20] (40) to (4-2);
\draw[bend left=30] (40) to (4-2);
\draw[bend right=30] (40) to (4-2);
\draw[bend left=40] (40) to (4-2);
\draw[bend right=40] (40) to (4-2);
\draw[bend left] (40) to (5-1);
\draw[bend right] (4-2) to (5-1);
\end{tikzpicture}\]
\end{definition}

\begin{lemma}\label{cdgamma}
Let $F$ be a $k$-rowed matrix that is a configuration of

\[\begin{bmatrix}
1 & \cdots & 1\\
 0 & \cdots & 0\\
 & K_{k-2} & \\
\end{bmatrix}.\]

Then $\mathfrak{C}_{k-2}, \mathfrak{D}_{k-2}, \Gamma_{k-1} \in \mathcal{H}_F$.
\end{lemma}

\begin{proof}
Let $A_F \in \avoid(m,r,\Sym(F))$. Note that $\mathfrak{C}_{k-2}, \mathfrak{D}_{k-2}, \Gamma_{k-1}$ can be all characterized by some permutation of the sequence $(k-2)\cdot(i,j), (i,f), (j,\ell)$, for some symbols $i\neq j$, and symbols $f, \ell$ not necessarily distinct nor different from $i$ or $j$.

Consider $A_F^S$ where $S$ is a sequence of symbol pairs with a permutation of $(k-2)\cdot(i,j), (i,f), (j,\ell)$ as a subsequence. Suppose $|A_F^S|>0$ for a contradiction. Then
\[
F(i,j) \prec \begin{bmatrix}
1 & \cdots & 1\\
0 & \cdots & 0\\
& K_{k-2} & \\
\end{bmatrix}(i,j) \prec
\begin{bmatrix}
i&j
\end{bmatrix}^{k-2} \times \begin{bmatrix}
i&f
\end{bmatrix} \times
\begin{bmatrix}
j&\ell
\end{bmatrix}\prec  \prod_{(xy) \in S} \begin{bmatrix}
x&y
\end{bmatrix} \prec A_F,
\] 
a contradiction, so $|A_F^S| = 0$ for all such sequences $S$. This implies that $\mathfrak{C}_{k-2}, \mathfrak{D}_{k-2}, \Gamma_{k-1} \in \mathcal{H}_F$. 
\end{proof}

\begin{lemma}\label{starlemma}
Let $F$ be a $k$-rowed matrix that is a configuration of

\[\begin{bmatrix}
1 & \cdots & 1\\ 
 & K_{k-1} & 
\end{bmatrix}.\]

Then $(k-1) \cdot S_{r-1}, \mathfrak{C}_{k-1}, \mathfrak{D}_{k-1} \in \mathcal{H}_F$, where $S_{r-1}$ is a star with $r-1$ rays.
\end{lemma}

\begin{proof}
First, $\mathfrak{C}_{k-1}, \mathfrak{D}_{k-1} \in \mathcal{H}_F$ due to the previous lemma. This is because
\[
F \prec \begin{bmatrix}
1 & \cdots & 1\\ 
& K_{k-1} & 
\end{bmatrix} \prec \begin{bmatrix}
1 & \cdots & 1\\
0 & \cdots & 0\\
& K_{k-1} & \\
\end{bmatrix}.
\]
Let $A_F \in \avoid(m,r,\Sym(F))$. Now let $S$ be a finite sequence of symbol pairs such that there exists a subsequence $S'$ of $S$ with $\textup{Graph}(S') \cong (k-1) \cdot S_{r-1}$. Denote by $S_{r-1}(j)$ the star graph with $r-1$ edges centered at vertex $j$. If $\textup{Graph}(S') = (k-1) \cdot S_{r-1}(j)$, where $j \neq 0$, then suppose for a contradiction that $|A_F^S| > 0$. Since $j \neq 0$, there exists a symbol $i$ such that $i<j$. We then get that
\[
F(i,j) \prec \begin{bmatrix}
1 & \cdots & 1\\ 
& K_{k-1} & 
\end{bmatrix}(i,j) \prec \begin{bmatrix}
i&j
\end{bmatrix}^{k-1} \times \begin{bmatrix}
j & *
\end{bmatrix} \prec \prod_{(xy) \in S'} \begin{bmatrix}
x&y
\end{bmatrix} \prec \prod_{(xy) \in S} \begin{bmatrix}
x&y
\end{bmatrix} \prec A_F,
\] 
where $*$ is any symbol other than $i$ or $j$. This is a contradiction since $A_F \in \avoid(m,r,\Sym(F))$. Thus $|A_F^S| = 0 = O(1)$. 

Now suppose that $\textup{Graph}(S') = (k-1) \cdot S_{r-1}(0)$. Then suppose for a contradiction that $|A_F^S| > 1$. Since $A_F^S$ is simple, there exists a non-zero symbol $j$ in $A_F^S$. Thus,
\[
F(0,j) \prec \begin{bmatrix}
1 & \cdots & 1\\ 
& K_{k-1} & 
\end{bmatrix}(0,j) \prec \begin{bmatrix}
0&j
\end{bmatrix}^{k-1} \times \begin{bmatrix}
j 
\end{bmatrix} \prec \prod_{(xy) \in S'} \begin{bmatrix}
x&y
\end{bmatrix} \prec \prod_{(xy) \in S} \begin{bmatrix}
x&y
\end{bmatrix} \prec A_F,
\] 
which is a contradiction. This shows that $|A_F^S| \leq 1 = O(1)$. Thus, $(k-1) \cdot S_{r-1} \in \mathcal{H}_F$. 
\end{proof}

\section{Computational Results}\label{sec:Computational Results}
\indent\indent We present a few examples of forbidden configurations $F$ for which an asymptotically tight bound for $\forb(m,r,\Sym(F))$ can be obtained using Lemma \ref{multisymbol to multicolumn}. Since constant rows provided some obstacles in lower bound constructions, we want to see their effects on the upper bound.

\begin{theorem}\label{symT2}
Let $T_2$ be the $2 \times 2$ triangular matrix. Then

\[\forb(m,r,\sym(T_2)) = \Theta(m^{r-1}).\]
\end{theorem}

\begin{proof}
The lower bound follows from Theorem \ref{lbound}. 

For the upper bound, we first wish to find $\mathcal{H}_{T_2}$. Since we cannot have an edge with multiplicity $k+\lceil\log_2(p)\rceil=2+0=2$, we avoid all double edges. 

By Lemma $\ref{starlemma}$, we know that $S_{r-1}, \mathfrak{C}_1, \mathfrak{D}_1 \in \mathcal{H}_{T_2}$. Note that $\mathfrak{C}_1 \cong P_4$, the path on 4 vertices, and $\mathfrak{D}_1 \cong C_3$, the cycle on 3 vertices. 

We observe that $\ex_1(r, \{S_{r-1}, P_4, C_3\} ) \leq r-1$. If a simple graph $G$ on $r$ vertices has at least $r$ edges, then $C_n$ is a subgraph of $G$ for some $n = 3, ..., r$. We must avoid $C_n$ for any $n = 3, ..., r$ since we must avoid $C_3$ and $P_4$, which is a subgraph of $C_n$ for $n=4,...,r$. 

Now we suppose that a simple graph $G$ on $r$ vertices avoids $S_{r-1}, P_4, C_3$ as subgraphs and has $r-1$ edges. Avoiding $C_3$ and $P_4$ implies that we must avoid cycles of any length, so $G$ is a forest. We can further deduce that $G$ is a tree since a forest on $r$ vertices with more than one connected component will have fewer than $r-1$ edges. Now avoiding $P_4$ implies that the diameter of $G$ is at most 3, so $G$ must be a star graph, which is impossible since $G$ cannot contain a star graph by assumption, so $\ex_1(r, \{S_{r-1}, P_4, C_3\} ) \leq r-2$. By Lemma \ref{multisymbol to multicolumn}, we get that $\forb(m,r,\Sym(T_2)) = O(m^{r-1})$, matching the lower bound.
\end{proof}

\begin{theorem}\label{1111K2}
	Let $k \geq 3$ and 
	\[F = \begin{bmatrix}
	1 & \cdots & 1\\
	& K_{k-1} & 
	\end{bmatrix}.\]
	
	Then $\forb(m,r,\sym(F)) = \Theta(m^{\max\{(k-1)(r-1), (k-2)\binom{r}{2}\}})$.
	
\end{theorem}

\begin{proof}
	We have $\forb(m,r,\sym(F)) = \Omega(m^{\max\{(k-1)(r-1), (k-2)\binom{r}{2}\}})$ due to the constructions from Theorem \ref{prod} and \ref{lbound}. 
	
	We want to show that $(k-2) \cdot K_r \in \mathcal{H}_F$. Let $A_F \in \avoid(m,r,\sym(F))$ and $S$ be a sequence of symbol pairs such that there exists $S'$ a subsequence of $S$ such that $\textup{Graph}(S') \cong (k-2) \cdot K_r$. Suppose that for a contradiction that $|A_F^S| \geq 2$, then we have $\begin{bmatrix}i & j\end{bmatrix} \prec A_F^S$ with $i < j$. Then since we have $(k-2) \cdot (ij)$ and $(j*)$ with $* \neq i$ as a subsequence of $S'$, we get that 
	
	\[F(i,j) = \begin{bmatrix}
	j & \cdots & j\\
	& K_{k-1}(i,j) &
	\end{bmatrix} \prec \begin{bmatrix}i & j\end{bmatrix} \times \begin{bmatrix}i & j\end{bmatrix}^{k-2} \times \begin{bmatrix}j & *\end{bmatrix} \prec \prod_{(xy) \in S'} \begin{bmatrix}
	x&y
	\end{bmatrix} \prec \prod_{(xy) \in S} \begin{bmatrix}
	x&y
	\end{bmatrix} \prec A_F, \]
	which is a contradiction, so $|A_F^S| < 2$. In other words, $|A_F^S| = O(1)$ so $(k-2) \cdot K_r \in \mathcal{H}_F$. 
	
	By Lemma \ref{starlemma}, we also get that $(k-1) \cdot S_{r-1}, \mathfrak{C}_{k-1}, \mathfrak{D}_{k-1} \in \mathcal{H}_F$. We must also avoid edges of multiplicity $k$ so we wish to calculate $\ex_{k-1}(r, \{(k-2) \cdot K_r, (k-1) \cdot S_{r-1}, \mathfrak{C}_{k-1}, \mathfrak{D}_{k-1} \})$. 
	
	Suppose that $G$ is a multigraph of maximum edge multiplicity at most $k-1$ on $r$ vertices with at least $(k-2)\binom{r}{2}$ edges that avoids $\{(k-2) \cdot K_r, (k-1) \cdot S_{r-1}, \mathfrak{C}_{k-1}, \mathfrak{D}_{k-1} \}$ as subgraphs. Let $G'$ be the underlying simple graph, that is each multiple edge is replaced by a single one. Then since $G$ cannot have $(k-2) \cdot K_r$ as a subgraph, there must exist  edges $\{i_1, j_1\},\{i_2, j_2\}, \ldots \{i_t, j_t\}  $ in $G$ with multiplicity at least $k-1$, but the maximum edge multiplicity is $k-1$ so the edges $\{i_{\ell}, j_{\ell}\}$ have multiplicities $k-1$. 
	
	Since $G$ avoids $\mathfrak{C}_{k-1}, \mathfrak{D}_{k-1}$ as subgraphs, at most one of the vertices $i_{\ell}$ and $j_{\ell}$ can be incident to another edge besides the $k-1$ edges between $i_{\ell}$ and $j_{\ell}$, that is all edges $\{i_1, j_1\},\{i_2, j_2\}, \ldots \{i_t, j_t\}  $ are pendant edges in the underlying simple graph $G'$. There are $t$ leaves among the vertices $\{i_1, j_1,i_2, j_2, \ldots i_t, j_t \} $, let their set be $L$. Then the number of edges of $G$ is at most $t(k-1)+(k-2)\binom{r-t}{2}$, since edges on the set $V(G)\setminus L$ have multiplicity at most $k-2$ and each vertex of $L$ is incident only with $k-1$ edges whose other endvertex is in   $V(G)\setminus L$. This implies the inequality
\begin{equation}\binom{r}{2}(k-2)\le t(k-1)+(k-2)\binom{r-t}{2},\end{equation}
        that is equivalent with
\begin{equation}\label{eq:krt}(k-2)[2rt-t^2-t]\le 2t(k-1).\end{equation}
Observe that $1\le t<r-1$. The latter inequality follows from the fact that if $t=r-1$, then $(k-1) \cdot S_{r-1}$ is a subgraph of $G$. So plugging in $t+2$ in place of $r$ in (\ref{eq:krt}) we obtain
\begin{equation}
  (k-2)[2(t+2)t-t^2-t]\le 2t(k-1),
\end{equation}
in simpler form
\begin{equation}\label{eq:kt}
  (k-2)[t^2+3t]\le 2t(k-1).
\end{equation}
Since $t\ge 1$, (\ref{eq:kt}) implies $(k-2)4t\le 2t(k-1)$, that is $2(k-2)\le k-1$, which is a contradiction if $k>3$. If $k=3$, then we must have equality everywhere, in particular $t^2+3t=4t$, so $t=1$ and $r=t+2=3$.
We must prove in this case that $\forb(m, 3, \sym(F)) = O(m^4)$. 	Suppose that $G$ is a multigraph of maximum edge multiplicity at most $2$ on $3$ vertices with at least $4$ edges that avoids $\{ K_3, 2 \cdot S_{2}, \mathfrak{C}_{2}, \mathfrak{D}_{2} \}$ as subgraphs. To avoid a $K_3$ and edges of multiplicity at least 3, $G$ must contain a $2 \cdot S_{2}$. Thus, such a $G$ is impossible. 
	
	All cases lead to contradictions, so  $\ex_{k-1}(r,\mathcal{H}_F) \leq \max\{(k-1)(r-1), (k-2)\binom{r}{2}\} -1$. Using Lemma \ref{multisymbol to multicolumn}, we show that $\forb(m, r, \sym(F)) = O(m^{\max\{(k-1)(r-1), (k-2)\binom{r}{2}\}})$. 
	
	The lower bounds and upper bounds agree so $\forb(m, r, \sym(F)) = \Theta(m^{\max\{(k-1)(r-1), (k-2)\binom{r}{2}\}})$. 
\end{proof}

\begin{theorem}\label{thm:4.3}
	Let $F$ be
	\[
	F = \begin{bmatrix}
	1 & 1 \\
	0 & 0 \\
	0 & 1 \\
	\end{bmatrix}.
	\]
	Then $\forb(m,r,\Sym(F)) = \Theta(m^{r-1})$. 
\end{theorem}

\begin{proof}
	Theorem \ref{lbound} provides the lower bound.
	
	First off, we see that $3 \cdot \{i, j\} \in \mathcal{H}_F$ by Lemma \ref{klog2p}. Also, by Lemma \ref{cdgamma}, we have that $\Gamma_2, \mathfrak{C}_1 \cong P_4, \mathfrak{D}_1 \cong C_3 \in \mathcal{H}_F$. 
	
	We claim that any graph $G$ with $r$ vertices and $\deg(v) \geq 1$ for every vertex $v$ is also in $\mathcal{H}_F$. Let $A_F \in \avoid(m,r,\Sym(F))$. Suppose for a contradiction that $|A_F^S| \geq 2$, where there exists a subsequence $S'\subseteq S$ such that $\textup{Graph}(S') \cong G$. Then since $A_F^S$ is simple, there exists $\begin{bmatrix}
	i & j
	\end{bmatrix} \prec A_F^S$ for some symbols $i \neq j$. However,
	\[
	F(i,j) \prec \begin{bmatrix}
	i & j
	\end{bmatrix} \times \begin{bmatrix}
	i & *
	\end{bmatrix}
	\times \begin{bmatrix}
	j & *
	\end{bmatrix} \prec \prod_{(xy) \in S'} \begin{bmatrix}
	x&y
	\end{bmatrix} \prec \prod_{(xy) \in S} \begin{bmatrix}
	x&y
	\end{bmatrix} \prec A_F,
	\]
	a contradiction, so $|A_F^S| < 2$ and thus $|A_F^S| = O(1)$. Therefore, $G \in \mathcal{H}_F$. 
	
	Now suppose we have a graph $G$ with at least $r-1$ edges with maximum edge multiplicity 2 and avoids everything in $\mathcal{H}_F$ as subgraphs. Suppose $G$ has $k$ edges of multiplicity 2. Since $\Gamma_2 \in \mathcal{H}_F$, this gives us $k$ connected components in $G$ each with 2 vertices and 2 edges. This leaves $r-2k$ vertices and $(r-1)-2k$ edges. Among these $r-2k$ vertices, we only have edges of multiplicity 1 but then the induced subgraph of these $r-2k$ vertices must be isomorphic to $S_{r-2k-1}$ because we avoid all cycles due to $C_3, P_4 \in \mathcal{H}_F$. However, $G \in \mathcal{H}_F$ since now we see each vertex have degree at least 1. 
	
	By Lemma \ref{multisymbol to multicolumn}, we have that $\forb(m,r,\Sym(F)) = O(m^{r-1})$ and thus $\forb(m,r,\Sym(F)) = \Theta(m^{r-1})$.
\end{proof}

\section{Block Matrices}\label{sec:blocks}

\indent\indent Let $\mathbf{i}_{p\times q}$ denote the constant matrix of dimensions $p\times q$ with $i$ in every entry. The purpose of this section is to see how large a matrix has to be before a configuration consisting of only 1 or 2 constant matrices must appear. The motivation for this section comes from the fact that constant matrices only have constant rows, which defy the product construction. We begin by looking at the smallest non-trivial case.

\begin{remark}
$\forb(m,r,\sym(\begin{bmatrix}
0\\1
\end{bmatrix})) = r$
\end{remark}

Clearly, we can only have columns containing only 1 symbol.

We can extend the previous remark by having the same matrix repeated some constant number of times. Interestingly, this increases the asymptotic bound from constant to linear. Also, notice how the asymptotic bound is unaffected by the number of columns in the matrix as long as it is greater than 1.

\begin{proposition} \label{01p}
If $q > 1$, then $\forb(m,r,\sym(\begin{bmatrix}
\mathbf{0}_{1\times q} \\ \mathbf{1}_{1\times q}
\end{bmatrix})) = \Theta(m)$. 
\end{proposition}

\begin{proof}
$I_m$ is a linear construction that avoids $\sym(\begin{bmatrix}
\mathbf{0}_{1\times q} \\ \mathbf{1}_{1\times q}
\end{bmatrix})$ and therefore we have the lower bound $\forb(m,r,\sym(\begin{bmatrix}
\mathbf{0}_{1\times q} \\ \mathbf{1}_{1\times q}
\end{bmatrix})) = \Omega(m)$.

Let $A \in \avoid(m,r,\sym(\begin{bmatrix}
\mathbf{0}_{1\times q} \\ \mathbf{1}_{1\times q}
\end{bmatrix}))$. We have that $A^{(ij)} \in \avoid(m-1, r, \mathbf{S}(\mathbf{0}_{1\times q}))$ since otherwise we would get $\begin{bmatrix}\mathbf{i}_{1\times q} \\ \mathbf{k}_{1\times q}
\end{bmatrix}, \begin{bmatrix}\mathbf{j}_{1\times q} \\ \mathbf{k}_{1\times q}
\end{bmatrix} \in A$ for some $0 \leq k \leq r - 1$, and we must have $i \neq k$ or $j \neq k$. But if $|A^{(ij)}| > r(q-1) + 1$, we have at least $q$ of some symbol in a row of $A^{(ij)}$, meaning $|A^{(ij)}| = O(1)$, which implies that $\forb(m,r,\sym(\begin{bmatrix}
\mathbf{0}_{1\times q} \\ \mathbf{1}_{1\times q}
\end{bmatrix})) = O(m)$. 
\end{proof}

From the restrictions of the $\Sym$, we have $\forb(m, r, \sym(\mathbf{0}_{p\times q})) = \forb(m, r, \sym(\mathbf{1}_{p\times q}))$. Again, interestingly, the asymptotic bound is unaffected by the number of columns. It is only affected by the number of rows in the constant matrix we want to avoid and the number of symbols.

\begin{theorem}
\label{constant}
If $r \geq 3$, then for $m \geq (p-1)(r-1)$, 
\[\forb(m,r,\Sym(\mathbf{0}_{p\times q})) \le \sum_{k=0}^{(p-1)(r-1)}\binom{m}{k}\sum_{\substack{k_1,...,k_{r-1}<p, \\ k_1+\cdots + k_{r-1} = k}}\binom{k}{k_1,...,k_{r-1}} + (q-1)(r-1)\binom{m}{p}.\]
Equality holds if $r-1\ge q-1$.
\end{theorem}

\begin{proof}
We will begin with the upper bound. If we had at least \[\sum_{k=0}^{(p-1)(r-1)}\binom{m}{k}\sum_{\substack{k_1,...,k_{r-1}<p, \\ k_1+\cdots + k_{r-1} = k}}\binom{k}{k_1,...,k_{r-1}} + (q-1)(r-1)\binom{m}{p} + 1\] columns, then at least $(q-1)(r-1)\binom{m}{p} + 1$ columns will have a majority symbol among the symbols not equal to $r$ that appears at least $p$ times in the column, since the first term is the total number of possible columns without this property. But after these $(q-1)(r-1)\binom{m}{p} + 1$ columns, there must exist $p$ rows such that at least $(q-1)(r-1) + 1$ columns has some symbol not equal to $r$ appearing $p$ times for each of the columns.

Then we can apply the Majority Principle to say that there is some  symbol $i\not=r$ that appears at least $q$ times within those $(q-1)(r-1) + 1$ columns and the same $p$ rows, which creates a $\mathbf{i}_{p \times q}$. Thus, 
\[\forb(m,r,\Sym(\mathbf{0}_{p\times q})) \leq \sum_{k=0}^{(p-1)(r-1)}\binom{m}{k}\sum_{\substack{k_1,...,k_{r-1}<p, \\ k_1+\cdots + k_{r-1} = k}}\binom{k}{k_1,...,k_{r-1}} + (q-1)(r-1)\binom{m}{p}.\]

The lower bound is given by the following construction. We take all possible columns that have at most $(p-1)(r-1)$ symbols not equal to $r$ with at most $p-1$ of any non-$r$ symbol. This gives us
\[
\sum_{k=0}^{(p-1)(r-1)}\binom{m}{k}\sum_{\substack{k_1,...,k_{r-1}<p, \\ k_1+\cdots + k_{r-1} = k}}\binom{k}{k_1,...,k_{r-1}}\]
columns. These columns cannot contribute to building something in $\Sym(\mathbf{0}_{p\times q})$ since there are not $p$ of any symbol not equal to $r$ in a column.

For $r\ge q$ we take every possible way to choose $p$ rows out of the $m$ rows, and for each possibility, we have $q-1$ columns that fill those $p$ rows with the same symbol not equal to $r$ and a different symbol not equal to $r$ in a different row to make the columns distinct, with the  symbol $r$ filling up the rest of the rows, for each of the $r-1$ symbols not equal to $r$. This gives us the remaining $(q-1)(r-1)\binom{m}{p}$ columns, which do not contain anything in $\Sym(\mathbf{0}_{p\times q})$ since there are no $q$ columns that contain $p$ of some symbol not equal to $r$.

The lower and upper bounds agree, implying an exact bound.
\end{proof}

The previous theorem shows that after $\Theta(m^{(p-1)(r-1)})$ columns, we are guaranteed a constant matrix configuration in some non-zero symbol. We are also interested to see how many columns it takes to form constant matrix configuration in any symbol. The lifting of the non-zero symbol restriction brings our asymptotic bound all the way down to $\Theta(1)$. 

\begin{theorem}
If $p \geq 2$ and $r \geq 3$, then for all $m \geq r^{r(q-1)}(p-1) + 1$, 
\[r\min(q-1,r)\le\forb(m,r,\mathbf{S}(\mathbf{0}_{p\times q})) \le r(q-1)\]
\end{theorem}

\begin{proof}
Suppose for some $A \in \avoid(m,r,\mathbf{S}(\mathbf{0}_{p\times q})) $, we have that $|A| \geq r(q-1) + 1$. Then in the first column, since there are at least $r^{r(q-1)}(p-1) + 1$ rows, there exists a majority symbol that appear at least $r^{r(q-1) - 1}(p-1) + 1$ times. Now within those $r^{r(q-1) - 1}(p-1) + 1$ rows, the second column has a majority symbol that appears at least $r^{r(q-1) - 2}(p-1) + 1$ times. We repeat this process until we get to the $r(q-1)^\text{th}$ column. Within the $r(p-1) + 1$ rows of the $r(q-1)^\text{th}$ column, we must have some symbol appearing at least $p$ times in the $(r(q-1) + 1)^\text{st}$ column.

At this point, we have $p$ rows such that all of the first $r(q-1) + 1$ columns has the same symbol repeated within those rows for each column. But within some symbol $i$ out of the $r$ symbols must be the repeated symbol in at least $q$ of the columns within these $p$ rows, producing a $\mathbf{i}_{p\times q}$. Thus, $\forb(m,r,\mathbf{S}(\mathbf{0}_{p\times q})) \leq r(q-1)$

The lower bound is given by the matrix with $\min(r,q-1)$ columns with $m-1$ entries of the same symbol, and a different symbol for the last entries to make the columns distinct, for each of the $r$ symbols. This gives $\forb(m,r,\mathbf{S}(\mathbf{0}_{p\times q})) \geq r\min(q-1,r)$.
\end{proof}

We now want to figure out the asymptotic bounds for avoiding a constant matrix stacked on top of another constant matrix of a different symbol, but first, we need a lemma to determine the number of columns a matrix needs to guarantee a large constant matrix of any symbol.

\begin{lemma}\label{cheatlemma}
If $m \geq r(p-1) + 1$ and $f(m)$ is some function of $m$, then
\[
\forb(m, r, \mathbf{S}(\mathbf{0}_{p\times f(m)})) \le r^{r(p-1)}(f(m)-1). 
\]
\end{lemma}

\begin{proof}
Let $A$ be an $(0,1,...,r-1)$-matrix with $r^{r(p-1)}(f(m)-1)+1$ columns, then in the first row of $A$, we have some majority symbol at least $\lceil\frac{r^{r(p-1)}(f(m)-1)+1}{r}\rceil = r^{r(p-1)-1}(f(m)-1)+1$ times. Now we repeat this process under these $r^{r(p-1)-1}(f(m)-1)+1$ columns. After $r(p-1) + 1$ rows, one of the $r$ symbols must appear at least $p$ times, creating a $p\times f(m)$ constant submatrix of some symbol. 
\end{proof}

We are now ready to prove the asymptotic bound for avoiding two constant matrices of different symbols in the same columns. Surprisingly, the asymptotic bound is the same as the one for avoiding a single constant matrix of a non-zero symbol. 

\begin{theorem}\label{onesontopofzeroes}
For a sufficiently large $m$, 
\[
\forb(m, r, \sym(\begin{bmatrix}
\mathbf{1}_{p_1 \times q} \\
\mathbf{0}_{p_0 \times q}
\end{bmatrix})) = \Theta(m^{(p-1)(r-1)}),
\]
where $1 < p = \max(p_0, p_1)$. 
\end{theorem}

\begin{proof}
The lower bound is given by Theorem \ref{lbound}.

For convenience, let
\[
f(m) = \sum_{k=0}^{(p-1)(r-1)}\binom{m}{k}\sum_{\substack{k_1,...,k_{r-1}<p, \\ k_1+\cdots + k_{r-1} = k}}\binom{k}{k_1,...,k_{r-1}}.
\]

Let $A$ be an $(0,1,...,r-1)$-matrix with $r^{r(p-1)}(f(m) +(q-1)(r-1)\binom{m}{p})+1$ columns. By Lemma \ref{cheatlemma}, we have in $A$ some symbol $i$ that creates a $\mathbf{i}_{p \times f(m) + (q-1)(r-1)\binom{m}{p}+1}$. We look in those $f(m) + (q-1)(r-1)\binom{m}{p}+1$ columns that contain the constant submatrix. We have $f(m)$ columns which contain at most $p-1$ of any of the $r-1$ non-$i$ symbol. Thus we have at least $(q-1)(r-1)\binom{m}{p} + 1$ more columns with at least $p$ of some non-$i$ symbol. This means that we have some $j \neq i$ such that a $\textbf{j}_{p\times q}$ appears within the same columns as the $\mathbf{i}_{p \times f(m) + (q-1)(r-1)\binom{m}{p}+1}$. This creates a matrix in $\sym(\begin{bmatrix}
\mathbf{1}_{p_1 \times q} \\
\mathbf{0}_{p_0 \times q}
\end{bmatrix})$ and we have the upper bound.
\end{proof}

Now we want to look at constant matrices laid side by side. 

\begin{remark}
$
\forb(m, r, \sym(\begin{bmatrix}
0 & 1 
\end{bmatrix})) = 1.
$
\end{remark}

We discovered that having two 1-rowed constant matrices laid side by side results in the same asymptotic bounds as having two 1-rowed constant matrices laid on top of each other. 

\begin{proposition}\label{01sidebyside}
If $p > 1$ or $q > 1$, then 
\[
\forb(m, r, \sym(\begin{bmatrix}
\mathbf{0}_{1\times p} & \mathbf{1}_{1\times q}
\end{bmatrix})) = \Theta(m).
\]
\end{proposition}

\begin{proof}
$I_m \in \avoid(m, r, \sym(\begin{bmatrix}
\mathbf{0}_{1\times p} & \mathbf{1}_{1\times q}
\end{bmatrix}))$ so $\forb(m, r, \sym(\begin{bmatrix}
\mathbf{0}_{1\times p} & \mathbf{1}_{1\times q}
\end{bmatrix})) = \Omega(m)$.

We observe that $\supp(\begin{bmatrix}
\mathbf{0}_{1\times p} & \mathbf{1}_{1\times q}
\end{bmatrix}) = \zeroone$, and $\forb(m, r, \sym(\zeroone)) = 1$, so by Theorem \ref{multiplicity}, we have $\forb(m, r, \sym(\begin{bmatrix}
\mathbf{0}_{1\times p} & \mathbf{1}_{1\times q}
\end{bmatrix})) = O(m)$.
\end{proof}

However, having blocks of constant matrices laid side by side results in a bigger asymptotic bound that having blocks of constant matrices laid on top of each other. 

\begin{theorem}\label{01blockssidebyside}
	If $p \geq 2$, 
\[
\forb(m, r, \sym(\begin{bmatrix}
\mathbf{0}_{p\times q_0} & \mathbf{1}_{p\times q_1}
\end{bmatrix})) = \Theta(m^{(p-1)\binom{r}{2}})
\]
\end{theorem}

\begin{proof}
We see that $\supp(\begin{bmatrix}
\mathbf{0}_{p\times q_0} & \mathbf{1}_{p\times q_1}
\end{bmatrix}) = \begin{bmatrix}
\mathbf{0}_{p\times 1} & \mathbf{1}_{p\times 1}
\end{bmatrix}$. So in order to get the lower bound, we first observe that $K_m^{p-1} \in \avoid(m, \begin{bmatrix}
\mathbf{0}_{p\times 1} & \mathbf{1}_{p\times 1}
\end{bmatrix})$, where $K_m^{p-1}$ denotes the $m\times +\binom{m}{p-1}$ simple $m$-rowed matrix of all columns with  $p-1$ 1's. Thus $\forb(m, \begin{bmatrix}
\mathbf{0}_{p\times 1} & \mathbf{1}_{p\times 1}
\end{bmatrix}) = \Omega(m^{p-1})$, and by Theorem \ref{lbound}, we get the lower bound $\forb(m, r, \sym(\begin{bmatrix}
\mathbf{0}_{p\times q_0} & \mathbf{1}_{p\times q_1}
\end{bmatrix})) = \Omega(m^{(p-1)\binom{r}{2}})$. 

By Theorem \ref{extendedextendedsauer}, we have $\forb(m, r, \sym(\begin{bmatrix}
\mathbf{0}_{p\times q_0} & \mathbf{1}_{p\times q_1}
\end{bmatrix})) = O(m^{(p-1)\binom{r}{2}})$ since it is the case that $\begin{bmatrix}
\mathbf{0}_{p\times q_0} & \mathbf{1}_{p\times q_1}
\end{bmatrix} \prec \max\{q_0, q_1\} \cdot K_p$. 
\end{proof}

The following table summarizes our results for block matrices: 

\begin{table}[!htbp]
\begin{center}
{\tabulinesep=1.2mm
\begin{tabu}{|c|c|c|}
\hline
Configuration $F$ & $\forb(m,r,\sym(F))$ & Proof \\ \hline

$\begin{bmatrix}
0\\1
\end{bmatrix}$ & $r$ & Remark in Section \ref{sec:blocks} \\ \hline

$\begin{bmatrix}
0 & \cdots & 0 \\
1 & \cdots & 1
\end{bmatrix}_{2\times q}$ & $\Theta(m)$ & Proposition \ref{01p} \\ \hline

$\begin{bmatrix}
0 & \cdots & 0 \\
\end{bmatrix}_{1\times q}$ & $\Theta(m)$ & Theorem \ref{constant}\\ \hline

$\begin{bmatrix}
0 & \cdots & 0 &
1 & \cdots & 1
\end{bmatrix}_{1\times (q_0 + q_1)}$ & $\Theta(m)$ & Proposition \ref{01sidebyside}\\ \hline

$\begin{bmatrix}
0 & \cdots & 0 \\
\vdots & \ddots & \vdots \\
0 & \cdots & 0
\end{bmatrix}_{p\times q}$ & $\Theta(m^{(p-1)(r-1)})$ & Theorem \ref{constant}\\ \hline

$\begin{bmatrix}
0 & \cdots & 0 \\
\vdots & \ddots & \vdots \\
0 & \cdots & 0 \\
1 & \cdots & 1 \\
\vdots & \ddots & \vdots \\
1 & \cdots & 1
\end{bmatrix}_{(p_0+p_1)\times q}$ & $\Theta(m^{(\max\{p_0, p_1\}-1)(r-1)})$ & Theorem \ref{onesontopofzeroes}\\ \hline

$\begin{bmatrix}
0 & \cdots & 0 & 1 & \cdots & 1 \\
\vdots & \ddots & \vdots & \vdots & \ddots & \vdots \\
0 & \cdots & 0 & 1 & \cdots & 1 \\
\end{bmatrix}_{p\times (q_0 + q_1)}$ & $\Theta(m^{(p-1)\binom{r}{2}})$ & Theorem \ref{01blockssidebyside}\\ \hline
\end{tabu}}
\caption{Asymptotic Bounds for Block Matrices}
\label{table:blocks}
\end{center}
\end{table}

\section{Further Directions}
Like with forbidden configurations of $(0,1)$-matrices, asymptotic bounds for the $\forb$ function are not known for many matrices. The generalization of forbidden configurations into multiple symbols introduces the interesting obstacle of constant rows, which we have not fully tackled yet. In particular, Theorem \ref{1111K2} only gives an upper bound for $\forb(m,r,\sym(F))$ for forbidden configurations $F$ that only have one constant row. To fully understand the effects of constant rows, we should consider the following matrix
\[F =
\begin{bmatrix}
1 &\cdots &1 \\
\vdots & \ddots & \vdots \\
1 & \cdots & 1 \\
0 &\cdots &0 \\
\vdots & \ddots & \vdots \\
0 & \cdots & 0 \\
& K_{k} &
\end{bmatrix}
\]
and the corresponding function $\forb(m, r, \Sym(F))$. 

We introduced a new lower bound construction with Theorem \ref{lbound} which deals with some of difficulties the product construction has with constant rows. However, just like the $(0,1)$ case, we do not know if the lower bound constructions we currently know of are sufficient. In other words, if $F$ is eligible for a lower bound with Theorem \ref{prod} or \ref{lbound}, is this lower bound an asymptotically tight bound?

\bibliographystyle{alpha}
\bibliography{multisymbol}

\end{document}